\documentclass[12pt, reqno]{amsart}
\usepackage{amsmath, amsthm, amscd, amsfonts, amssymb, graphicx, color}
\usepackage{mathrsfs}
\usepackage{graphicx}
\usepackage{caption}
\usepackage{subcaption}
\usepackage[bookmarksnumbered, colorlinks, plainpages]{hyperref}
\hypersetup{colorlinks=true,linkcolor=red, anchorcolor=green, citecolor=cyan, urlcolor=red, filecolor=magenta, pdftoolbar=true}

\newtheorem{theorem}{Theorem}[section]
\newtheorem{Lemma}[theorem]{Lemma}

\newtheorem{Corollary}[theorem]{Corollary}

\newtheorem{example}[theorem]{Example}

\theoremstyle{remark}
\newtheorem{remark}[theorem]{\bf{Remark}}
\numberwithin{equation}{section}
\allowdisplaybreaks
\begin{document}

\title[]{On posinormality of weighted composition-differentiation operators on $H^2(\mathbb{D})$}

\author[G. Hait, S. Ojha, N. Ghosh, R. Birbonshi]{Gour Hait, Sarita Ojha, Nirupam Ghosh, Riddhick Birbonshi}
	
\address[Hait]{Department of Mathematics, Indian Institute of Engineering Science and Technology, Shibpur, Howrah 711103, West Bengal, India}
\email{chandragour.math@gmail.com}
	
\address[Ojha]{Department of Mathematics, Indian Institute of Engineering Science and Technology, Shibpur, Howrah 711103, West Bengal, India}
\email{sarita.ojha89@gmail.com}

\address[Ghosh] {Department of Mathematics, Indian Institute of Engineering Science and Technology, Shibpur, Howrah 711103, West Bengal, India}
\email{nirupamghoshmath@gmail.com}

\address[Birbonshi]{Department of Mathematics, Jadavpur University, Kolkata 700032, West Bengal, India.}
\email{riddhick.math@gmail.com}



\begin{abstract} 
	 In this article, the posinormality and coposinormality of weighted composition-differentiation operators on Hardy space $H^2(\mathbb{D})$ are investigated. It is observed that while a composition-differentiation operator $D_{\phi,n}$ fails to be posinormal, the weighted composition-differentiation operator $D_{\psi,\phi,n}$ can be posinormal for specific choices of $\psi, \phi$. Some necessary conditions are obtained for posinormality and coposinormality of the operator $D_{\psi,\phi,n}$. Furthermore, the adjoint formula for this operator is derived which also helped us to examine some results regarding posinormality of this operator.
\end{abstract}

\maketitle	

\section{Introduction}
Let $\mathbb{D}$ denote the open unit disk in the complex plane. The Hardy space $H^2 (\mathbb{D})$ is the Hilbert space consisting of all analytic functions $f(z)=\sum\limits_{n=0}^\infty a_nz^n$ on $\mathbb{D}$ equipped with the norm
\begin{equation*}
    ||f||_2=\sqrt{\sum_{n=0}^\infty|a_n|^2}<\infty.
\end{equation*}
For $f(z)=\sum\limits_{n=0}^\infty a_nz^n$ and $g(z)=\sum\limits_{n=0}^\infty b_nz^n$ in $H^2(\mathbb{D})$, their inner product is defined as 
\begin{equation*}
    \langle f,g \rangle=\sum_{n=0}^\infty a_n\overline{b_n}.
\end{equation*}
Let $H^\infty(\mathbb{D})$ denote the space of all bounded analytic functions on $\mathbb{D}$, with $||f||_\infty=\mbox{sup}\{|f(z)|:z\in\mathbb{D\}}$.\\

For any $w\in\mathbb{D}$, the function 
\begin{equation}\label{K_w}
    K_w(z)=\frac{1}{1-\overline{w}z}, \ z\in\mathbb{D}
\end{equation} 
is the reproducing kernel function for the point-evaluation in $H^2(\mathbb{D})$, i.e., 
$$\langle f,K_w\rangle=f(w)$$ 
for any $f\in H^2(\mathbb{D})$. In a similar manner, for any natural number $n$ and $w\in\mathbb{D}$, the function 
\begin{equation}\label{K_w ^n}
   K_w^{[n]}(z)=\frac{n!z^n}{(1-\overline{w}z)^{n+1}}, \ z\in\mathbb{D} 
\end{equation}
serves as the reproducing kernel for point-evaluation of the $n$-th derivative of $f$ at the point $w$, i.e.,
    $$\langle f,K_w^{[n]}\rangle=f^{(n)}(w)$$ 
for any $f\in H^2(\mathbb{D})$~\cite[Theorem 2.16]{cowen1995composition}. 

For an analytic map $\phi:\mathbb{D}\to\mathbb{D}$, the composition operator $C_{\phi}:H^2(\mathbb{D})\to H^2(\mathbb{D})$ is defined by 
\begin{equation*}
    (C_{\phi}f)(z)=f(\phi(z))
\end{equation*}
for all $f\in H^2(\mathbb{D})$. From Corollary 3.7 of \cite{cowen1995composition}, it follows that every composition operator is bounded on $H^2(\mathbb{D})$. For a function $\psi\in H^{\infty}(\mathbb{D})$, the Toeplitz operator $T_{\psi}$ on $H^2(\mathbb{D})$ is defined by
\begin{equation*}
    T_{\psi}f=\psi\cdot f
\end{equation*}
for all $f\in H^2(\mathbb{D})$. Every Toeplitz operator is bounded on  $H^2(\mathbb{D})$ with $||T_\psi||=||\psi||_{\infty}$ (see \cite[Theorem 5]{vukotic2003analytic}). For notational convenience, we denote the operator $T_{\psi}$ by $T_z$ when $\psi(z)=z$ on $\mathbb{D}$.

Though the differentiation operator $D(f)=f'$ is unbounded on $H^2(\mathbb{D})$, there are many analytic self-maps $\phi:\mathbb{D}\to\mathbb{D}$ such that the operators $C_\phi D$ or $DC_\phi$, defined as 
\begin{equation*}
    C_\phi D f(z)= f'(\phi(z)) \mbox{ and } DC_\phi f(z)=f'(\phi (z))\phi'(z)
\end{equation*}
for all $z\in\mathbb{D}$, are bounded on $H^2(\mathbb{D})$. Such operators were first studied in \cite{hibschweiler2005composition} and \cite{ohno2006products}, and were subsequently investigated in \cite{fatehi2020composition}. Since $D$ is unbounded, it makes sense to write $C_\phi D$ as a single operator $D_\phi$ as follows:
\begin{equation}\label{D_phi}
    D_{\phi}f=f'\circ\phi
\end{equation}
for all $f\in H^2(\mathbb{D})$. This operator is referred to as the composition-differentiation operator. The operator $D_{\phi}$ is  guaranteed to be bounded if $||\phi||_{\infty}<1$ (see \cite[Theorem 3.3]{ohno2006products}), although there exist examples for which $D_{\phi}$ is bounded or compact when $||\phi||_{\infty}=1$.


For any analytic map $\psi:\mathbb{D}\to\mathbb{C}$ (not necessarily belonging to $H^\infty(\mathbb{D})$, the weighted composition-differentiation operator $D_{\psi,\phi}$ on $H^2(\mathbb{D})$ is defined as follows:
\begin{equation}\label{D_psi,phi}
    D_{\psi,\phi}(f)=\psi \cdot (f'\circ\phi).
\end{equation}
It is clear that if  $\psi\in H^{\infty}(\mathbb{D})$ and the operator $D_\phi$  is bounded on $H^2(\mathbb{D})$, then $D_{\psi,\phi}=T_{\psi}D_{\phi}$ is guaranteed to be bounded (for example, take $||\phi||_{\infty}<1$).
The operators $C_\phi D$ and $DC_\phi$ can each be expressed in the form $D_{\psi,\phi}$, where $\psi\equiv1$ in the first case and $\psi\equiv\phi'$ in the second.

The weighted composition-differentiation operator of order $n$ is defined as
\begin{equation}\label{D_psi,phi,n}
    D_{\psi,\phi,n}f=\psi\cdot (f^{(n)}\circ\phi)
\end{equation}
where $n\in\mathbb{N}$. The operator $D_{\psi,\phi,n}$ is bounded on $H^2(\mathbb{D})$ whenever $\psi\in H^{\infty}(\mathbb{D})$ and $D_{\phi,n}$ is bounded on $H^2(\mathbb{D})$ (such as $||\phi||_{\infty}<1$, for details see \cite{gunatillake2023weighted}). This operator generalizes many well known operators. Some of them are listed as follows:
\begin{enumerate}
    \item If $n=1$ and $\psi\equiv 1$, $D_{\psi,\phi,n}$ reduces to composition-differentiation operator $D_\phi$ (as in \eqref{D_phi}). 
    \item For $n=1$, it becomes the weighted composition-differentiation operator $D_{\psi,\phi}$ defined in \eqref{D_psi,phi}.
\end{enumerate}  


On the other other hand, as defined by Rhaly in \cite{rhaly1994posinormal}, a bounded linear operator $T$ on a Hilbert space $H$ is said to be posinormal if there exists a positive operator $P$ such that 
$TT^*=T^*PT$.  The operator $P$ is referred to as an interrupter of $T$. As noted by Rhaly, if $V$ is an isometry (so that $V^*V =I$) and $T$ is posinormal with interrupter $P$, then $VTV^*$ is posinormal with interrupter $VPV^*$. Consequently, posinormality is unitary invariant. The operator $T$ is said to be coposinormal if its adjoint $T^*$ is posinormal. Using results of Douglas \cite{douglas1966majorization}, Rhaly obtained a number of equivalent characterizations of posinormal operators.
\begin{theorem}\cite{rhaly1994posinormal}\label{Rahly}
    For  $T\in B(H)$, the space of all bounded linear operators on a Hilbert space $H$, the following
statements are equivalent:
\begin{enumerate}
 \item $T$ is posinormal.
 \item $Range(T)\subseteq Range(T^*)$.
 \item $TT^*\leq\lambda^2 T^*T$ for some $\lambda\geq 0$.
 \item there exists an operator $A\in B(H)$ such that $T=T^*A$.
\end{enumerate}
\end{theorem}
\begin{Corollary}\cite{rhaly1994posinormal}\label{kernel}
    If $T$ is posinormal, then $\mbox{Kernel}(T)\subseteq \mbox{Kernel}(T^*)$; in particular $\mbox{Kernel}(T)$ is a reducing subspace for the posinormal operator $T$.
\end{Corollary}

Composition operators and composition-differentiation operators have attracted considerable attention in the literature. Fatehi and Hammond \cite{fatehi2020composition, fatehi2021normality} studied fundamental operator-theoretic properties, including the adjoint, norm, self-adjointness, and aspects of normality of the  operators defined in \eqref{D_phi} and \eqref{D_psi,phi}. Their work was later extended by Lo and Loh \cite{Lo} to weighted composition-differentiation operator of order $n$. More recently, Bourdon and Thompson \cite{bourdon2023posinormal} have obtained some characterizations for both posinormal and coposinormal composition operators on $H^2(\mathbb{D})$. Motivated by these works, we focus in this paper on the posinormal and coposinormal behavior of weighted composition–differentiation operators of order $n$ on $H^2(\mathbb{D})$.

This paper is organized as follows.\\
In Section 2,  we show that the composition-differentiation operator $D_{\phi,n}$ (i.e., $\psi\equiv1$ in \eqref{D_psi,phi,n}) can not be posinormal for any analytic self-map $\phi$ on $\mathbb{D}$. We also derive some necessary conditions for the weighted composition-differentiation operator $D_{\psi,\phi,n}$ to be posinormal and coposinormal.\\
In Section 3, we obtain an adjoint formula for $D_{\psi,\phi,n}$ for a linear fractional $\phi$ and some particular choice of $\psi$. By using this formula, we remark some results on posinormality of this operator.

\section{Posinormality and coposinormality of $D_{\psi,\phi,n}$}

In this section, we discuss posinormality and coposinormality of $D_{\psi,\phi,n}$ on $H^2(\mathbb{D})$. Throughout this section, we assume that $\psi$ is not identically zero. We start with the following result:

\begin{Lemma}\label{prop1}
Let $\psi:\mathbb{D}\to \mathbb{C}$, $\phi:\mathbb{D}\to \mathbb{D}$ be analytic functions and $n\in\mathbb{N}$ such that $D_{\psi,\phi
,n}$ is bounded on $H^2(\mathbb{D})$. If $f\in\textnormal{Range}(D_{\psi,\phi
,n}^*)$, then $f^{(m)}(0)=0$ for every $0\leq m<n$, where $f^{(0)}\equiv f$.
\end{Lemma}
\begin{proof}
    Let $f\in\mbox{Range}(D_{\psi,\phi
,n}^*)$. Then there exists an element $g\in H^2(\mathbb{D})$ such that $D_{\psi,\phi,n}^*\,g=f$. Let $m$ be a non-negative integer such that $0\leq m<n$. 
Now we consider the following two cases.\\
Case 1 $(m=0)$: From \eqref{K_w}, $K_0(z)=1$ for all $z\in\mathbb{D}$. Now
\begin{eqnarray*}
        f(0)= f^{(0)} (0) &=&\langle f,K_0\rangle\\
        &=&\langle D_{\psi,\phi,n}^*g, \ K_0\rangle\\
        &=&\langle g,\ D_{\psi,\phi,n}K_0\rangle\\
        &=&\left\langle g,\ \psi\cdot \left(K_0\right)^{(n)}\circ\phi\right\rangle\\
        &=&\langle g, 0\rangle\\
        &=& 0
    \end{eqnarray*}
Case 2 $(m>0)$: From \eqref{K_w ^n}, it follows that $(K_0^{[m]})(z)=m!z^m$ for $m\geq 1$. Therefore $\left(K_0^{[m]}\right)^{(n)}(z)=0$. 
    Now \begin{eqnarray*}
        f^{(m)}(0)&=&\langle f,K_0^{[m]}\rangle\\
        &=&\langle D_{\psi,\phi,n}^*g, \ K_0^{[m]}\rangle\\
        &=&\langle g,\ D_{\psi,\phi,n}K_0^{[m]}\rangle\\
        &=&\left\langle g,\ \psi\cdot \left(K_0^{[m]}\right)^{(n)}\circ\phi\right\rangle\\
        &=&\langle g, 0\rangle\\
        &=& 0
    \end{eqnarray*}
whenever $0\leq m <n$.
\end{proof}


Fatehi and Hammond \cite{fatehi2021normality} have given some necessary conditions that a weighted composition-differentiation operator $D_{\psi,\phi}$ can be normal on $H^2(\mathbb{D})$ as follows.
\begin{theorem}\cite{fatehi2021normality}\label{normality condition}
    Suppose $D_{\psi,\phi}$ is normal on $H^2(\mathbb{D})$. Then the function $\psi$ has the following properties :
\begin{enumerate}
    \item $\psi(0)=0$.
    \item $\psi(w)\neq0$ for any $w\in\mathbb{D}\setminus\{0\}$.
    \item $\psi'(0)\neq0$.
\end{enumerate}
Moreover, the map $\phi$ must be univalent.
\end{theorem}
As a consequence of the above result, they noted that an unweighted composition-differentiation operator $D_{\phi}$ cannot be normal on $H^2(\mathbb{D})$. An extension of this result is given by Lo and Loh \cite{Lo} as given below: 
\begin{theorem}\cite{Lo}\label{normality for n}
    Let $\psi:\mathbb{D}\to\mathbb{C}$ and $\phi:\mathbb{D}\to\mathbb{D}$ be analytic functions such that $D_{\psi,\phi,n}:H^2(\mathbb{D})\to H^2(\mathbb{D})$ is normal for $n\in\mathbb{N}$. Then
    \begin{enumerate}
        \item $\psi(0)=0$ and $\psi(z)\neq0$ for every $z\in\mathbb{D}\setminus\{0\}$; and
        \item $\psi^{(k)}(0)=0$ for $k=1,2,\ldots,n-1$ and $\psi^{(n)}(0)\neq0$.
    \end{enumerate}
\end{theorem}

In light of the above Theorem, it is clear that $D_{\phi,n}$ is not normal for any $\phi$ and for any $n\in\mathbb{N}$. We now present a broader version of this result.

\begin{theorem}\label{D_phi pos}
    The unweighted composition-differentiation operator $D_{\phi,n}$ on $H^2(\mathbb{D})$ of order $n$ can not be posinormal for any self-map $\phi$ on $\mathbb{D}$ and for any $n\in\mathbb{N}$.
\end{theorem}
\begin{proof}
    Suppose that $D_{\phi,n}$ is bounded for some $\phi$ and for some $n\in\mathbb{N}$. Let $g(z)=a$ and $\displaystyle f(z)=\frac{az^n}{n!}$ where $a$ is a nonzero complex number. Clearly $g,f\in H^2(\mathbb{D})$. Also, for all $z\in\mathbb{D}$, 
    \begin{equation*}
        (D_{\phi,n} f)(z)=f^{(n)}(\phi(z))=a=g(z), \mbox{ i.e., } D_{\phi,n} f=g 
    \end{equation*}
    which implies $g\in\mbox{Range}(D_{\phi,n})$. 
    Since $g(0)\neq 0$, from Lemma \ref{prop1} (i.e., for $\psi\equiv1$), we get $g\notin\mbox{range}(D_{\phi,n}^*)$. So $\mbox{Range}(D_{\phi,n})\nsubseteq \mbox{Range}(D_{\phi,n}^*)$. From Theorem \ref{Rahly}, it follows that $D_{\phi,n}$ can not be posinormal for any $\phi$ and for any $n\in\mathbb{N}$. 
\end{proof}


Bourdon and Thompson \cite{bourdon2023posinormal} obtained the following necessary condition on posinormality of the composition operator $C_{\phi}$ as,
\begin{theorem}\cite[Corollary 1.4]{bourdon2023posinormal}\label{C_phi posinormal}
    If $C_\phi$ is posinormal and $\phi$ is linear fractional, then $\phi(\beta)=0$ for some $\beta\in\mathbb{D}$.
\end{theorem}
However, Theorem \ref{D_phi pos} demonstrates that a comparable condition is unattainable for $D_{\phi,n}$. Interestingly, although the unweighted composition-differentiation operator $D_{\phi,n}$ cannot be posinormal, the weighted version $D_{\psi,\phi,n}$ can achieve posinormality for some $\psi$. Since every normal operator is necessarily posinormal, Theorem \ref{normality condition} and Theorem \ref{normality for n} further imply that a weighted composition-differentiation operator can be posinormal for some appropriate choice of $\psi$. Also, in \cite[Proposition 4]{fatehi2021normality}, Fatehi and Hammond provided a characterization of all normal operators $D_{\psi, \phi}$ on $H^2(\mathbb{D})$ satisfying $\phi(0)=0$, which is further generalized by Lo and Loh in \cite[Theorem 4.5]{Lo} for the operator $D_{\psi,\phi,n}$. This naturally leads us to investigate the existence of posinormal operators that are strictly non-normal. Consider the following example.
\begin{example}\label{example1}
     Let $\psi(z)=z^2$ and $\phi(z)=az$, where $0<|a|<1$ and $z\in\mathbb{D}$. As $\|\phi\|_{\infty}<1$, so $D_{\psi,\phi}$ is bounded. As a direct consequence of Theorem \ref{normality condition}, we conclude that $D_{\psi,\phi}$ is not normal as it violates the third condition.\\
     Let $\{e_k\}$ be the orthonormal basis of $H^2(\mathbb{D})$, where $e_k(z)=z^k, \ z\in\mathbb{D},\ k\geq0$. Let 
     \begin{equation*}
         f(z)=\displaystyle \sum_{k=0}^\infty a_kz^k=\sum_{k=0}^\infty a_ke_k (z)
     \end{equation*}
     and $S_k$ denote the partial sum of this series. Therefore $\lim\limits_{k\to\infty}(D_{\psi,\phi}S_k)(z)=(D_{\psi,\phi}f)(z)$. This gives
     \begin{eqnarray*}
         (D_{\psi,\phi}f)(z) &=& \lim_{k\to\infty}D_{\psi,\phi}\left(\sum_{r=0}^k a_rz^r\right)\\
    &=& \lim_{k\to\infty}\sum_{r=0}^ka_rD_{\psi,\phi}z^r\\
         &=& \lim_{k\to\infty}(a_0\cdot0+a_1z^2+2aa_2z^3+3a^2a_3z^4+\ldots+ka^{k-1}a_kz^{k+1})\\
         &=&\sum_{k=0}^\infty w_ka_kz^{k+1}
     \end{eqnarray*}
     where $\{w_k\}$ is the sequence with $w_0=0$ and $w_k=ka^{k-1}, \ k\geq1$. Note that $\displaystyle \sum_{k=0}^\infty|w_ka_k|^2 <\infty$ and  $\mbox{Range}(D_{\psi,\phi})=\overline{\mbox{Span}\{z^2,z^3,\ldots\}}$. \\
     Thus, $D_{\psi,\phi}$ acts as a bounded unilateral weighted shift operator on $H^2{(\mathbb{D})}$ with the weight sequence $\{w_k\}$. So, its adjoint $D_{\psi,\phi}^*$ on the orthonormal basis $\{e_k\}$ is given by     
     $$D_{\psi,\phi}^*e_k=\begin{cases}
     0, & k=0,1\\
         \overline{w_{k-1}}\ e_{k-1}, & k\geq 2.
     \end{cases}$$
    Therefore,
     $\mbox{Range}(D_{\psi,\phi}^*)=\overline{\mbox{Span}\{z,z^2,z^3,\ldots\}}$. So, $\mbox{Range}(D_{\psi,\phi})\subseteq\mbox{Range}(D_{\psi,\phi}^*)$. Hence from Theorem \ref{Rahly}, it follows that $D_{\psi,\phi}$ is posinormal.\\
\end{example}

Consequently, we provide the following necessary condition for posinormality:





\begin{theorem}\label{pos D_psi phi_n}
    Let $\psi:\mathbb{D}\to \mathbb{C}$ and $\phi:\mathbb{D}\to \mathbb{D}$ be analytic functions. Suppose $D_{\psi,\phi,n}$ is posinormal on $H^2(\mathbb{D})$ for some $n\in\mathbb{N}$. Then $\psi^{(m)}(0)=0$ for every $0\leq m<n$.
\end{theorem}
\begin{proof}
     Let $\displaystyle f(z)=\frac{z^n}{n!}$ for all $ z\in\mathbb{D}$. Then
\begin{equation*}
    (D_{\psi,\phi,n}f)(z) =\psi(z)\cdot f^{(n)}(\phi(z))=\psi(z), \mbox{ i.e, } D_{\psi,\phi,n}f=\psi
\end{equation*}
which gives $\psi\in\textnormal{Range}(D_{\psi,\phi,n})$. Since $D_{\psi,\phi,n}$ is posinormal, we must have $\psi\in\textnormal{Range}(D_{\psi,\phi,n}^*)$ from Theorem \ref{Rahly}. Therefore by Lemma \ref{prop1}, we have $\psi^{(m)}(0)=0$ for every $0\leq m<n$.  
\end{proof}

A natural question can arise for the converse part of the above theorem, which we answer in the last section of the present article. To proceed further, the following result \cite{hu2023generalized} will be used.
\begin{Lemma}\cite[Lemma 1]{hu2023generalized}\label{lemman}
    Suppose $\psi:\mathbb{D}\to \mathbb{C}$ and $\phi:\mathbb{D}\to \mathbb{D}$ be analytic functions. Let $n$ be a positive integer and $w\in\mathbb{D}$. If $D_{\psi,\phi,n}$ is bounded on $H^2(\mathbb{D})$, then
    \begin{eqnarray*}
D_{\psi,\phi,n}^*K_w &=& \overline{\psi(w)}K_{\phi(w)}^{[n]}.
\end{eqnarray*}
\end{Lemma}

Now, consider the following example where $\psi(0)\neq0$. 
\begin{example}\label{ex1}
    Let $\psi(z)=\lambda \ (\neq0)$ and $\phi(z)=az$, where $0<|a|<1$ and $z\in\mathbb{D}$. Since $||\phi||_{\infty}<1$, $D_{\psi,\phi}$ is bounded. From the first condition of Theorem \ref{normality condition}, it follows that $D_{\psi,\phi}$ is not normal.\\
    Consider the orthonormal basis $\{e_k\}$ of $H^2(\mathbb{D})$, where $e_k(z)=z^k,~z\in\mathbb{D}, k\in \mathbb{N}\cup \{0\}$. Let $f\in H^2(\mathbb{D})$ with $\displaystyle f(z)=\sum_{m=0}^\infty b_mz^m$. 
   From the proof of Theorem 3.4 of \cite{han2022some}, it follows that  $D_{\psi,\phi}^* z^k=(k+1)\overline{\lambda a^k}z^{k+1}~~,~~k\in\mathbb{N}$.\\    
   Now for $k=0$, $D_{\psi,\phi}^*1=(D_{\psi,\phi}^*K_0)(z)=\overline{\psi(0)}K_{\phi(0)}^{[1]}=\overline{\lambda}z$. Therefore 
   $$D_{\psi,\phi}^* z^k=(k+1)\overline{\lambda a^k}z^{k+1}~~,~~k\geq0.$$
Thus, $D_{\psi,\phi}^*$ acts as a bounded unilateral weighted shift operator with the nonzero weight sequence $\{w_k\}$ where $w_k=(k+1)\overline{\lambda a^k} , \ k\in\mathbb{N}\cup\{0\}$. Hence from \cite[Proposition 1.1]{rhaly1994posinormal}, we can say that $D_{\psi,\phi}^*$ is posinormal, i.e., $D_{\psi,\phi}$ is coposinormal.
\end{example}

The above example shows that although an unweighted composition-differentiation operator $D_{\phi}$ fails to be posinormal, but it can still be coposinormal (i.e., $\lambda=1$ in Example \ref{ex1}). To proceed further, we need the following lemma.

\begin{Lemma}\label{ker=poly}
    Let $\psi:\mathbb{D}\to\mathbb{C}$ be an analytic function and $\phi$ be a nonconstant analytic self-map of $\mathbb{D}$. Then $\mbox{Kernel}(D_{\psi,\phi,n})$ consists exactly all polynomials in $\mathbb{C}$ of degree $\leq n-1$.
\end{Lemma}
\begin{proof}
    Clearly, every polynomial of degree less than $n$ belongs to $\mbox{Kernel}(D_{\psi,\phi,n})$. Conversely, let $f\in\mbox{Kernel}(D_{\psi,\phi,n})$. Then
$$(D_{\psi,\phi,n}f)(z)=\psi(z)\cdot f^{(n)}(\phi(z))=0~~~~~~\mbox{for all}~z\in\mathbb{D}.$$
    Since $\psi$ is not identically $0$ on $\mathbb{D}$ and $\phi$ is nonconstant analytic function on $\mathbb{D}$ such that $\phi(\mathbb{D})\subset\mathbb{D}$, by open mapping theorem and identity theorem, it follows that $f^{(n)}(z)=0$ for all $z\in\mathbb{D}$. This implies that $f$ is a polynomial of degree $\leq n-1$ on $\mathbb{D}$.  
\end{proof}

Next, consider the following example.
\begin{example}\label{ex2}
    Let $\psi(z)=z^2$ and $\phi(z)=\frac{z^2}{2}$, where $z\in\mathbb{D}$. Then clearly $\psi$ and $\phi$ do not satisfy the conditions given in Theorem \ref{normality condition}. Hence $D_{\psi,\phi}$ is bounded but not normal.\\
    Let $\{e_k\}$ be the orthonormal basis of $H^2(\mathbb{D})$, where $e_k(z)=z^k$,  $z\in\mathbb{D}, k\in \mathbb{N}\cup \{0\}$. Then $(D_{\psi,\phi}e_k)(z)=\psi(z)e_k'(\phi(z))=\frac{k}{2^{k-1}}z^{2k}.$ 
    Hence 
    \begin{eqnarray*}
        \mbox{Range}(D_{\psi,\phi}) &\subseteq& \overline{\mbox{Span}\{z^2, z^4,\ldots\}}\\
        \mbox{i.e., } \ \overline{\mbox{Span}\{z^2, z^4,\ldots\}}^{\perp} &\subseteq& [\mbox{Range}(D_{\psi,\phi})]^{\perp}\\
        \mbox{i.e., } \ \overline{\mbox{Span}\{1,z, z^3,\ldots\}} &\subseteq& \mbox{Kernel}(D_{\psi,\phi}^*).
    \end{eqnarray*}
    Let $g(z)=z^3, ~z\in\mathbb{D}$. Then $g\in\mbox{Kernel}(D_{\psi,\phi}^*)$, but $g\notin \mbox{Kernel}(D_{\psi,\phi})$ as the kernel of $D_{\psi,\phi}$ contains only constant functions by Lemma \ref{ker=poly} for $n=1$. Hence $D_{\psi,\phi}$ is not coposinormal by Corollary \ref{kernel}.
\end{example}

From Example \ref{ex1}, it follows that $D_{\psi,\phi}$ is coposinormal even if $\psi(0)\neq 0$. On the other hand, Example \ref{ex2} shows that $D_{\psi,\phi}$ fails to be coposinormal when $\psi(0)=0=\psi'(0)$ and $\phi$ is not injective. Motivated by these two examples, here we derive the following necessary conditions on the coposinormality of $D_{\psi,\phi,n}$ on $H^2(\mathbb{D})$.

\begin{theorem}\label{pro4}
   Let $\psi:\mathbb{D}\to\mathbb{C}$ be an analytic function and $\phi$ be a nonconstant analytic self-map of $\mathbb{D}$. Suppose that $D_{\psi,\phi,n}$ is coposinormal on $H^2(\mathbb{D})$ for $n\in\mathbb{N}$. Then $\psi$ and $\phi$ satisfy the following conditions :
   \begin{enumerate}
       \item\label{pro41} $\psi(w)\neq 0$ for any $w\in\mathbb{D}\setminus\{0\}$.
       \item\label{pro43} $0$ can not be the zero of $\psi$ of multiplicity $>n$.
       \item\label{pro44} The map $\phi$ must be injective.
   \end{enumerate}
\end{theorem}
\begin{proof}
    For any $w\in\mathbb{D}$, we have 
    \begin{equation}\label{matrix}
    D_{\psi,\phi,n}(K_w)=\psi\cdot(K_w^{(n)}\circ\phi)=\frac{n!\overline{w}^n\psi}{(1-\overline{w}\phi)^{n+1}}
    \end{equation}
    where $K_w^{(n)}$ denotes the $n$-th derivative of $K_w$ as in \eqref{K_w}.
    \begin{enumerate}
        \item Suppose that $\psi(w)=0$ for some $w\in\mathbb{D}\setminus\{0\}$. Then from Lemma \ref{lemman},
        \begin{equation*}
            D_{\psi,\phi,n}^*(K_w)=0, \ \mbox{ i.e., } \  K_w\in\textnormal{Kernel}(D_{\psi,\phi,n}^*).
        \end{equation*}
        Since $D_{\psi,\phi,n}$ is coposinormal, so from Corollary \ref{kernel}, it is evident that $\textnormal{Kernel}(D_{\psi,\phi,n}^*)\subseteq \textnormal{Kernel}(D_{\psi,\phi,n})$. This gives $K_w\in \textnormal{Kernel}(D_{\psi,\phi,n})$. From \eqref{matrix}, it follows that $\overline{w}^n\psi\equiv0$. Since $\psi$ is not identically zero, we conclude that $w=0$, a contradiction. Therefore $\psi(w)\neq 0$.

    \item For any $f\in H^2(\mathbb{D})$ and $w\in\mathbb{D}$, we have 
    \begin{eqnarray*}
        \langle f, D_{\psi,\phi,n}^*K_w^{[n]}\rangle &=&\langle D_{\psi,\phi,n}f, \ K_w^{[n]}\rangle\\
&=&\langle\psi\cdot(f^{(n)}\circ\phi),\ K_w^{[n]}\rangle\\
        &=&\left(\psi\cdot(f^{(n)} \circ \phi)\right)^{(n)}(w)
    \end{eqnarray*}
    where $K_w^{[n]}$ is given in \eqref{K_w ^n}. Now, by Leibnitz rule for successive differentiation, we obtain
    \begin{equation}\label{n>0}
        \langle f, D_{\psi,\phi,n}^* K_w^{[n]}\rangle=\sum_{i=0}^n {n \choose i}\psi^{(i)}(w)(f^{(n)}\circ\phi)^{(n-i)}(w)
    \end{equation}
    If possible, let `$0$' is a zero of $\psi$ of multiplicity $>n$. Then $\psi^{(i)}(0)=0$ for $i=0, 1, 2, \ldots, n$. Putting $w=0$ in \eqref{n>0}, we get 
    \begin{equation*}
        \langle f,D_{\psi,\phi,n}^* K_0^{[n]}\rangle=0 \ \mbox{ for all $f\in H^2(\mathbb{D})$}.
    \end{equation*}
     This gives $D_{\psi,\phi,n}^*K_0^{[n]}=0$, i.e., $K_0^{[n]}\in\mbox{Kernel}(D_{\psi,\phi,n}^*)$. Again applying Corollary \ref{kernel}, we have 
     \begin{equation*}
         D_{\psi,\phi,n}K_0^{[n]}=0,  \mbox{ i.e., } \psi\cdot\left((K_0^{[n]})^{(n)}\circ\phi\right)=0
     \end{equation*}
     This gives $(n!)^2\psi\equiv0$ which implies $\psi\equiv0$, a contradiction. Hence the result follows. 
    \item To prove that $\phi$ is injective, let us consider two distinct points $w_1$ and $w_2$ in $\mathbb{D}$ satisfying $\phi(w_1)=\phi(w_2)$. Initially, we assume that $w_1$ and $w_2$ are both non-zero. Consider the function $\displaystyle h=\frac{K_{w_{1}}}{\overline{\psi(w_1)}}-\frac{K_{w_{2}}}{\overline{\psi(w_2)}}$. By the condition \eqref{pro41} of the ongoing theorem, $h$ is well-defined and belongs to $H^2(\mathbb{D})$. Now from Lemma \ref{lemman}, it follows that  
\begin{equation*}
    D_{\psi,\phi,n}^*(h) = D_{\psi,\phi,n}^*\left(\frac{K_{w_{1}}}{\overline{\psi(w_1)}}-\frac{K_{w_{2}}}{\overline{\psi(w_2)}}\right) = K_{\phi(w_1)}^{[n]}-K_{\phi(w_2)}^{[n]} =0.
\end{equation*}
Therefore $h\in\mbox{Kernel}(D_{\psi,\phi,n}^{*})$.  Since $D_{\psi,\phi,n}$ is coposinormal, so $h\in \mbox{Kernel}(D_{\psi,\phi,n})$. By Lemma \ref{ker=poly}, we conclude that $h$ is a polynomial of degree less than or equal to $n-1$ on $\mathbb{D}$, which is a contradiction to our construction of the function $h$.\\
Next suppose $w_1=0$, and that $\phi(0)=\phi(w_2)$. Consider two disjoint open neighbourhoods, say $N_0$ of $0$ and $N_{w_2}$ of $w_2$. Since $\phi$ is a nonconstant analytic self-map of $\mathbb{D}$, hence, $\phi(N_0)$ and $\phi(N_{w_2})$ are two open sets in $\mathbb{D}$ containing $\phi(0)$ and $\phi(w_2)$, respectively. As $\phi(0)=\phi(w_2)$, then $\phi(N_0)$ and $\phi(N_{w_{2}})$ must intersect. Let $z_0\in \mathbb{D}$ be such that 
\begin{equation*}
   z_0  \in \phi(N_0)\cap \phi(N_{w_{2}}) \mbox{ and } z_0\neq\phi(0).
\end{equation*} 
Then there exist distinct non-zero elements $w_3\in N_0$ and $w_4\in N_{w_2}$ satisfying $\phi(w_3)=z_0= \phi (w_4)$. A similar reasoning again yields a contradiction from first part. Therefore, $\phi$ must be injective.
    \end{enumerate}
\end{proof}


\section{Remarks on posinormality via the adjoint of $D_{\psi,\phi,n}$}
In this section, we derive the adjoint of the operator $D_{\psi,\phi,n}$. Utilizing this representation, we establish several necessary conditions for its posinormality. \\
Consider the following functions
\begin{equation}\label{phiz sigma}
    \phi(z)=\frac{az+b}{cz+d} \mbox{ and } \sigma(z)=\frac{\overline{a}z-\overline{c}}{-\overline{b}z+\overline{d}}, \ \ z\in\mathbb{D}.
\end{equation}
The map $\phi$ is nonconstant if and only if $ad-bc\neq 0$. Cowen \cite{cowen1988linear} has shown that if $\phi:\mathbb{D}\to\mathbb{D}$ satisfies $ad-bc=1$, then $\sigma$ maps $\mathbb{D}$ into itself. Also from \cite{fatehi2020composition}, $||\sigma||_{\infty}<1$ whenever $||\phi||_{\infty}<1$.
Cowen's adjoint formula for composition operator $C_\phi$ with linear fractional $\phi$ is available in \cite{cowen1988linear} as follows: 
\begin{theorem}\label{Cowen adjoint}
    If $\displaystyle\phi(z)=\frac{az+b}{cz+d}$ is an analytic self-map of $\mathbb{D}$ with $ad-bc=1$. Then on $H^2(\mathbb{D})$, 
    \begin{equation*}
        C_\phi^*=T_gC_\sigma T_h^*
    \end{equation*}
    where $\displaystyle g(z)=\frac{1}{(-\overline{b}z+\overline{d})}$ and $h(z)= cz+d$ are in $H^{\infty}(\mathbb{D})$, and $\sigma(z)=\displaystyle\frac{\overline{a}z-\overline{c}}{-\overline{b}z+\overline{d}}$ is an analytic self-map on $\mathbb{D}$.
\end{theorem}
Though in Theorem \ref{Cowen adjoint}, Cowen requires $ad-bc=1$, but it is enough to take $ad-bc\neq0$. Also, Fatehi and Hammond in \cite{fatehi2020composition} have provided a formula involving the adjoint of $D_\phi$ for the nonconstant self-map of $\mathbb{D}$ defined by $\phi (z)=rz$ where $r$ is a real number. In Example 3 of \cite{fatehi2021normality}, they have given a formula for $D_{\psi,\phi}^*$ where $\psi(z)=z$ and $\phi(z)=az$, where $a\in\mathbb{D}\setminus\{0\}$ and $z\in\mathbb{D}$. Recently, Lo and Loh \cite{Lo} have provided a formula related to $D_{\psi,\phi,n}^*$ in the case $\psi\equiv sK_{\sigma(0)} ^{[n]}$, where $s\in\mathbb{C}$, $\phi$ and $\sigma$ are defined in \eqref{phiz sigma}.

Now, we derive similar type of formula as in Theorem \ref{Cowen adjoint} for the weighted composition-differentiation operator $D_{\psi,\phi,n}$ for particular choice of $\psi(z)$ and a linear fractional self-map $\phi$.

\begin{theorem}\label{Cowen}
    Let $\psi(z)=z^n\psi_1(z)$ for $z\in\mathbb{D}, \ \psi_1\in H^{\infty}(\mathbb{D})$ and $\phi, \sigma$ be nonconstant analytic self-maps of $\mathbb{D}$ as defined in \eqref{phiz sigma} with $\|\phi\|_{\infty}<1$. Then on $H^2(\mathbb{D})$,    $$D_{\psi,\phi,n}^*=T_gD_{\sigma,n}T_h^*$$
where $n\in\mathbb{N},\ \displaystyle g(z)=\frac{z^n}{(-\overline{b}z+\overline{d})^{n+1}}$ and $h(z)=\psi_1(z)(cz+d)^{n+1}$ are in $H^\infty(\mathbb{D})$.
\end{theorem}
\begin{proof}
It is easy to check that the functions $g(z)$ and $h(z)$ are in $H^\infty(\mathbb{D})$. From the given conditions, the operators $D_{\psi,\phi,n}$ and $D_{\sigma,n}$ are bounded on $H^2(\mathbb{D})$. Starting from the right hand side, we have
   \begin{eqnarray*}
       (T_gD_{\sigma,n} T_h^*K_w)(z)&=& g(z)D_{\sigma,n}\left(\overline{h(w)}K_w(z)\right)\\
       &=& \overline{h(w)}g(z)D_{\sigma,n} (K_w (z))\\
       &=& \overline{h(w)}g(z)K_w^{(n)}(\sigma(z))\\
       &=&\overline{h(w)}g(z)\frac{n!\overline{w}^n}{(1-\overline{w}\sigma(z))^{n+1}}\\
       &=& \overline{h(w)}g(z)\frac{n!\overline{w}^n}{\left(1-\overline{w}\displaystyle\left(\frac{\overline{a}z-\overline{c}}{-\overline{b}z+\overline{d}}\right)\right)^{n+1}}\\
       &=& \overline{h(w)}g(z)\frac{n!\overline{w}^n(-\overline{b}z+\overline{d})^{n+1}}{(-\overline{b}z+\overline{d}-\overline{aw}z+\overline{cw})^{n+1}}\\
       &=&\overline{h(w)}g(z)\frac{n!\overline{w}^n(-\overline{b}z+\overline{d})^{n+1}}{(\overline{cw}+\overline{d})^{n+1}\left\{1-\left(\frac{\overline{aw}+\overline{b}}{\overline{cw}+\overline{d}}\right)z\right\}^{n+1}}\\
       &=&\overline{h(w)}g(z)\frac{n!\overline{w}^{n}(-\overline{b}z+\overline{d})^{n+1}}{(\overline{cw}+\overline{d})^{n+1}\left(1-\overline{\phi(w)}z\right)^{n+1}}\\
        &=&\frac{\overline{w}^n\ \overline{\psi_1(w)} \ n!z^n}{(1-\overline{\phi(w)}z)^{n+1}}\\
        &=&\frac{\overline{\psi(w)}\ n!z^n}{(1-\overline{\phi(w)}z)^{n+1}}\\
        &=&\overline{\psi(w)}K_{\phi(w)}^{[n]}(z)\\
        &=&D_{\psi,\phi,n}^*(K_w)(z)
   \end{eqnarray*}
   follows from Lemma \ref{lemman}. Thus we have,
   $$D_{\psi,\phi,n}^*(K_w)=T_gD_{\sigma,n} T_h^*(K_w)~~~~~~\mbox{for any }~~w\in\mathbb{D}.$$
   Since the span of the reproducing kernel functions $\{K_w:w\in\mathbb{D}\}$ forms a dense subspace of $H^2(\mathbb{D})$ and the above two expressions agree on them, we can say that $D_{\psi,\phi,n}^{*}=T_gD_{\sigma,n} T_h^*$.  
\end{proof}

\begin{theorem}\label{phi(beta)=0}
    Let $\psi(z)=\lambda z^n$ for $z\in \mathbb{D}$, $n\in\mathbb{N}$, $\lambda\in\mathbb{C}\setminus\{0\}$ and $\phi$ be a nonconstant analytic self-map on $\mathbb{D}$ defined as in \eqref{phiz sigma} with $\|\phi\|_{\infty}<1$. Then the range of $D_{\psi,\phi,n}^{*}$ contains a function of the form $f(z)=\mu z^n$ for some $\mu\in \mathbb{C}$ if and only if $\phi(\beta)=0$ for some $\beta\in\mathbb{D}$.
\end{theorem}
\begin{proof}
    Suppose $\phi(\beta)=0$ for some $\beta\in\mathbb{D}$. Then from Lemma \ref{lemman}, we have
    \begin{eqnarray*}
        (D_{\psi,\phi,n}^{*}K_\beta)(z)&=&\overline{\psi(\beta)}K_{\phi(\beta)}^{[n]}(z)\\
        &=&\overline{\psi(\beta)}K_0^{[n]}(z)\\
        &=&\overline{\psi(\beta)}\ n!z^n\\
        &=&\mu z^n 
    \end{eqnarray*}
    where $\mu=\overline{\psi(\beta)}\ n! = \overline{\lambda\beta^n}\ n!$.\\
    Conversely, let $f(z)=\mu z^n\in \mbox{Range}(D_{\psi,\phi,n}^{*})$. Suppose $\phi$ and $\sigma$ are defined as in \eqref{phiz sigma}. So, if $b=0$, then $\phi(0)=0$. Let us consider $b\neq0$. In this case, our aim is to show that the zero $-\displaystyle\frac{b}{a}$ must lie in $\mathbb{D}$.\\
    By Theorem \ref{Cowen} (taking $\psi_1\equiv \lambda)$, there exists a function $\rho\in H^2(\mathbb{D})$ such that 
    \begin{eqnarray*}
        (T_gD_{\sigma,n} T_h^*\rho)(z) &=& \mu z^n~~,~z\in\mathbb{D}\\
        \mbox{i.e., } \ \displaystyle\frac{z^n}{(-\overline{b}z+\overline{d})^{n+1}}(D_{\sigma,n} T_h^{*}\rho)(z) &=& \mu z^n\\
        \mbox{i.e., } \ (D_{\sigma,n} T_h^{*}\rho)(z) &=& \mu(-\overline{b}z+\overline{d})^{n+1}.
    \end{eqnarray*}
    Let $q=T_h^{*}\rho$. So, the above equation reduces to 
    \begin{eqnarray*}
        (D_{\sigma,n} q)(z)&=& \mu(-\overline{b}z+\overline{d})^{n+1}\\
        \mbox{i.e., } \ q^{(n)}(\sigma(z)) &=&\mu(-\overline{b}z+\overline{d})^{n+1}
    \end{eqnarray*}
    Let $u=\sigma(z)$ where $\sigma(z)$ is defined as in \eqref{phiz sigma}. Then $u=\displaystyle\frac{\overline{a}z-\overline{c}}{-\overline{b}z+\overline{d}}$ and thus $z=\displaystyle\frac{\overline{c}+\overline{d}u}{\overline{a}+\overline{b}u}$. Hence, we get
    \begin{align*}
        q^{(n)}(u)&=\mu\left\{-\overline{b}\left(\displaystyle\frac{\overline{c}+\overline{d}u}{\overline{a}+\overline{b}u}\right)+\overline{d}\right\}^{n+1}\\
        &=\mu\left(\displaystyle\frac{-\overline{bc}-\overline{bd}u+\overline{ad}+\overline{bd}u}{\overline{a}+\overline{b}u}\right)^{n+1}\\
        &=\mu\left(\displaystyle\frac{\overline{ad}-\overline{bc}}{\overline{a}+\overline{b}u}\right)^{n+1}
    \end{align*}
    A straightforward calculation yields
    \begin{equation*}
        q(u) =(-1)^n\displaystyle\frac{\mu(\overline{ad}-\overline{bc})^{n+1}}{n!\overline{b}^n(\overline{a}+\overline{b}u)}+P_{n-1}(u) \  \mbox{for all $u\in\mbox{Range}(\sigma)$}
    \end{equation*}
 where $P_{n-1}$ is a polynomial in $u$ of degree $n-1$ arising from the constants of integration. Define $\xi$ as,
 \begin{equation*}
     \xi(u)=(-1)^n\displaystyle\frac{\mu(\overline{ad}-\overline{bc})^{n+1}}{n!\overline{b}^n(\overline{a}+\overline{b}u)}+P_{n-1}(u) \ \mbox{ for } \ u\in\mathbb{D}\setminus \left\{-\frac{\overline{a}}{\overline{b}}\right\}.
 \end{equation*}
  Thus, $q$ and $\xi$ agree on the non-empty open subset $\sigma(\mathbb{D)}$ of $\mathbb{D}$. Since $q$ is analytic on $\mathbb{D}$ and $\xi$ is analytic on $\mathbb{D}$ except possibly at $-\displaystyle\frac{\overline{a}}{\overline{b}}$, by identity theorem, the function $q$ and $\xi$ must agree on $\mathbb{D}$ except possibly at this point.\\
    However, $-\displaystyle\frac{\overline{a}}{\overline{b}}$ is a pole of $\xi~(\mbox{since }\overline{ad}-\overline{bc} ~~\mbox{is non-zero)}$ and therefore $-\displaystyle\frac{\overline{a}}{\overline{b}}$ can not lie in $\mathbb{D}$. If $-\frac{\overline{a}}{\overline{b}}\in\partial \mathbb{D}$, then $\displaystyle\left|\frac{\overline{a}}{\overline{b}}\right|=1\implies \left|\frac{\overline{b}}{\overline{a}}\right|=1$ and consequently $q$ would have Maclaurin series
    $$q(z)=\displaystyle\frac{(-1)^n\mu(\overline{ad}-\overline{bc})^{n+1}}{n!\overline{ab}^n}\sum_{k=0}^\infty(-1)^k\left(\frac{\overline{b}}{\overline{a}}\right)^k z^k     +P_{n-1}(z).$$
    Since $\displaystyle\sum_{k=0}^\infty\left|\frac{\overline{b}}{\overline{a}}\right|^{2k}=\infty$, so $q$ would not belong to $H^2(\mathbb{D})$, a contradiction. Thus $-\displaystyle\frac{\overline{a}}{\overline{b}}$ must lie outside the closed unit disk. So, the zero $-\displaystyle\frac{b}{a}$ of $\phi$ lies in $\mathbb{D}$.
    \end{proof}

As a consequence of the Theorem \ref{phi(beta)=0}, we have the following result.
\begin{Corollary}\label{Corolarry on pos}
         Let $\psi(z)=\lambda z^n$ for $z\in\mathbb{D}$, $n\in\mathbb{N}$, $\lambda\in\mathbb{C}\setminus\{0\}$ and $\phi$ be a nonconstant analytic self-map of $\mathbb{D}$ defined by \eqref{phiz sigma} with $\|\phi\|_{\infty}<1$. If $D_{\psi,\phi,n}$ is posinormal on $H^2(\mathbb{D})$, then $\phi(\beta)=0$ for some $\beta\in\mathbb{D}$.
    \end{Corollary}\label{cl2}
    \begin{proof}
        Let $f(z)=\displaystyle\frac{z^n}{n!}$ for all $z\in\mathbb{D}$. Then 
        \begin{equation*}
            (D_{\psi,\phi,n}f)(z)=\psi(z)\cdot f^{(n)}(\phi(z))=\psi(z)\cdot 1=\psi(z).
        \end{equation*}
        This gives $\psi\in\textnormal{Range}(D_{\psi,\phi,n})$. Since $D_{\psi,\phi,n}$ is posinormal, hence by Theorem \ref{Rahly},
        $\psi(z)=\lambda z^n\in\textnormal{Range}(D_{\psi,\phi,n}^*)$. Therefore, by  Theorem \ref{phi(beta)=0}, we must have $\phi(\beta)=0$ for some $\beta\in\mathbb{D}$.
    \end{proof}


A similar condition for posinormality is obtained for $C_{\phi}$ in Theorem \ref{C_phi posinormal}. From Corollary \ref{Corolarry on pos}, we derive the following result :
    \begin{Corollary}\label{cl16}
        Let $\psi(z)=\lambda z^n$ for $z\in\mathbb{D}$, $n\in\mathbb{N}$, $\lambda\in\mathbb{C}\setminus\{0\}$ and $\phi(z)=az+b$ with $|a|+|b|<1$ and $|a|<|b|$. Then $D_{\psi,\phi,n}$ is not posinormal.
    \end{Corollary}
    \begin{remark}
         For posinormality of $D_{\psi,\phi,n}$, we obtain two necessary conditions as established in Theorem \ref{pos D_psi phi_n} and Corollary \ref{Corolarry on pos}. 
       Let $\psi(z)=2z$ and $\displaystyle\phi(z)=\frac{z+2}{5} , z\in\mathbb{D}$. Clearly, $\psi(0)=0$. From Corollary \ref{cl16}, we conclude that $D_{\psi,\phi}$ is not posinormal. Therefore, the converse of Theorem \ref{pos D_psi phi_n} is not true in general.
    \end{remark}

\section{Declarations}
\textit{Acknowledgements:} Mr. Gour Hait would like to thank UGC, Govt. of India for the financial support (NTA Ref. No. 231610156338) in the form of fellowship.\\
\textit{Author Contributions:} All the authors contributed equally to this manuscript and reviewed it. \\
 \textit{Data Availability :} No datasets were generated or analysed during the current study. \\
\textit{Conflict of interest:} There is no competing interest.\\

\bibliographystyle{amsplain}
\bibliography{reference}

\end{document}